\documentclass[12pt,letterpaper]{amsart}
\usepackage{amssymb,amsmath,amscd,eucal,epsfig}
\def\<{\langle}                     
\def\>{\rangle}                     

\newcommand{\ben}{\begin{enumerate}}
\newcommand{\een}{\end{enumerate}}

\theoremstyle{plain}
\newtheorem{theorem}{Theorem}[section]
\newtheorem{proposition}{Proposition}[section]
\newtheorem{lemma}{Lemma}[section]
\newtheorem{corollary}{Corollary}[section]
\newtheorem{remark}{Remark}[section]

\theoremstyle{definition}
\newtheorem{definition}{Definition}[section]

\numberwithin{equation}{section}
\begin{document}

\begin{center}
{\textbf{{Convexity of Picture Fuzzy Multisets}}}\ \\ \ \\

{Taiwo O. Sangodapo},\\
Department of Mathematics,\\ University of Ibadan, Ibadan, Nigeria\\toewuola77@gmail.com
\end{center}
\ \\ \  \\

\noindent {\textbf{Abstract:}} 
In this paper, the notion of convexity of picture fuzzy multisets was introduced and some of their properties were presented after studying the concept of picture fuzzy multisets. 
\; \\ \; \\

\noindent {\textbf{Keywords}}: Fuzzy set, Convex set, Convex Fuzzy Set, Picture Fuzzy Set, Picture Fuzzy multisets
\; \\ \; \\

\noindent \textbf{MSC Classification:} 03E72, 35E10, 52A20
\ \\ \ \\

\section{Introduction} In 1965, Zadeh \cite{z} introduced convexity of fuzzy sets as an extension of classical convex sets. In 1980, Lowen \cite{l} extended Zadeh's work by introducing the concept of affine fuzzy sets in order to study convex fuzzy sets in greater details and investigated the properties of both convex and affine fuzzy sets using Euclidean points. After major contribution to Zadeh's work on convex fuzzy sets, many researchers have also studied the concept, such as, Liu \cite{l2} 1985 established some properties of convex fuzzy sets. Yang \cite{y} 1988, established the relationship among three types of convex fuzzy sets. In 2002, Yang and Yang \cite{yy} were able to characterised convexity of fuzzy sets via semicontinuous conditions. Peng \cite{p} 2010, proposed two new definitions of convex fuzzy sets, i.e., convex (s, t]-fuzzy sets via relation between fuzzy points and fuzzy subsets, and convex R-fuzzy sets via implication operators of fuzzy logic. In 2010, Cheng et al \cite{c} put forward the concept of (s, t]-intuitionistic convex fuzzy sets via cut set of intuitionistic fuzzy sets and neighborhood relations between fuzzy point and intuitionistic fuzzy set. In 2017, Sangodapo and Ajayi \cite{sa1} contributed to Lowen's work by establishing some properties of convex fuzzy sets using fuzzy points.

In 1986, Atanassov \cite{a1} initiated the concept of intuitionistic fuzzy sets (IFSs). In 2012, Huang \cite{h} generalised convex fuzzy sets to convex intuitionistic fuzzy sets and obtained equal characteristics in terms of cut sets. In 2018, Das and Mukhlalsah \cite{dm} also studied intuitionistic fuzy sets and developed intuitionistic concex fuzzy sets, established various theorems and illustrated with examples. In 2021, Sangodapo and Ajayi \cite{sa2} extended the Huang's work to affine intuitionistic fuzzy sets and also established their characteristics.

Cuong and Kreinovich \cite{ck} in 2013, generalised Zadeh and Atanassov's works by putting forward the theory of Picture fuzzy sets (PFSs). This theory is a new concept for computational intelligence problems to handle an important notion of neutrality degree that was lacking in IFSs theory. In 2014, Cuong \cite{c3} investigated some characteristics of PFSs, introduced distance measure and defined convex combination between two PFSs. Son \cite{s} in 2016, introduced a generalised picture distance measure and applied it to establish an Hierarchical Picture Clustering. In 2017, Dutta and Ganju \cite{dg} studied decomposition theorems of PFSs, defined $(\alpha, \gamma, \beta)$-cut of PFSs, and extension principle for PFSs and obtained their properties. PFSs has been extensively studied and applied, see \cite{dg, dp, g, tkc} for details.

Yagar in 1986 \cite{ya}, put forward the notion of fuzzy multisets (FMs). In 2013, Shinoj and Sunil \cite{ss1} initiated intuitionistic fuzzy multiset (IFMSs). In \cite{cfs}, Cao et al introduced picture fuzzy multisets (PFMSs) as a generalisation of FM and IFMS, \cite{ya} and \cite{ss1}, respectively and also as an extension of PFSs. In 2022. Sangodapo \cite{s} established the notion of convexity of picture fuzzy sets and obtained some properties of them. For more extension of PFMSs see \cite{d1, d2, d3}. Picture fuzzy multiset has also been extended to group theory, see \cite{t2, d4, sg, sy, p, tb}

In this paper, we study the work of Sangodapo \cite{s1} and extend it to the convexity of picture fuzzy multisets and investigated some of their properties.
The organisations of the paper is as follows. In section 2, we give necessary definitions and preliminary ideas of picture fuzzy multisets. Section 3 introduces some properties of convex picture fuzzy set and investigates its properties.  
\ \\

\section{PRELIMINARIES}
In this section, we state some basic definitions related to convex fuzzy sets, picture fuzzy sets and picture fuzzy multisets.

\begin{definition} \cite{ck}
A picture fuzzy set $A$ of $X$ is an object of the form $$A = \lbrace (x, \sigma_{A}(x), \tau_{A}(x), \gamma_{A}(x))| x \in X \rbrace,$$ where the functions $$\sigma_{A}: X \rightarrow [0, 1],~ \tau_{A}: X \rightarrow [0, 1]~ \text{and}~ \gamma_{A}: X \rightarrow [0, 1]$$ 
are called the positive, neutral and negative membership degrees of $x \in A$, respectively, and $\sigma_{A}, \tau_{A}, \gamma_{A}$ satisfy $$0 \leq \sigma_{A}(x) + \tau_{A}(x) + \gamma_{A}(x) \leq 1,~ \forall x \in X.$$ For each $x \in X$, $1 - (\sigma_{A}(x) + \tau_{A}(x) + \gamma_{A}(x))$ is called the refusal membership degree of $x \in A$.  
\end{definition}

\begin{definition} \cite{ck}
Let $A$ and $B$ be two PFSs. Then, the inclusion, equality, union, intersection and complement are defined as follow:
\begin{itemize}
\item $A \subseteq B$ if and only if for all $x \in X$, $\sigma_{A}(x) \leq \sigma_{B}(x)$, $\tau_{A}(x) \leq \tau_{B}(x)$ and $\gamma_{A}(x) \geq \gamma_{B}(x).$
\item $A = B$ if and only if $A \subseteq B$ and $B \subseteq A.$ 
\item $A \cup B = \lbrace ( x, \sigma_{A}(x) \vee \sigma_B(x), \tau_A(x) \wedge \tau_B(x)), \gamma_A(x) \wedge \gamma_B(x)) | x \in X \rbrace.$
\item $A \cap B = \lbrace ( x, \sigma_{A}(x) \wedge \sigma_B(x), \tau_A(x) \wedge \tau_B(x)), \gamma_A(x) \vee \gamma_B(x)) | x \in X \rbrace.$
\item $\overline{A} = \lbrace (x, \gamma_{A}(x), \tau_{A}(x), \sigma_{A}(x))| x \in X \rbrace.$
\end{itemize}
\end{definition}

\begin{definition} \cite{z}
A FS $A$ of $X$ is called convex fuzzy set if for all $x, y \in X$ and $\lambda \in [0, 1]$,
$$\sigma_{A} [\lambda x + (1 - \lambda) y] \geq \sigma_{A}(x) \wedge \sigma_{A}(y)$$
\end{definition}

\begin{definition} \cite{ck} \label{a1}
Let $A,~ B$ be PFSs on $X$ and $\lambda \in [0, 1]$. Then, the convex combination of $A~ \text{and}~ B$ is defined as $$C_{\lambda}(A,~B) = \lbrace (x, \sigma_{C_{\lambda}}(x), \tau_{C_{\lambda}}(x), \gamma_{C_{\lambda}}(x))| x \in X \rbrace$$ where $$\sigma_{C_{\lambda}}(x) = \lambda \sigma_{A}(x) + (1 - \lambda) \sigma_{B}(x),$$ $$\tau_{C_{\lambda}}(x) = \lambda \tau_{A}(x) + (1 - \lambda) \tau_{B}(x),$$ $$\gamma_{C_{\lambda}}(x) = \lambda \gamma_{A}(x) + (1 - \lambda) \gamma_{B}(x) $$ $\forall~ x \in X.$   
\end{definition}

\begin{remark} \label{i}
Definition \ref{a1} can be rewritten as a convex combination of two points of a PFS $A$ on $X$. That is;
$$PCC(A) = \lbrace ((x, y), (\sigma; \lambda)(x, y), (\tau, ; \lambda)(x, y), (\gamma; \lambda)(x, y))| x, y \in E \rbrace$$ for a PFS $A$ and $\lambda \in [0, 1]$ where
$$(\sigma; \lambda)(x, y) = (1- \lambda) \sigma_{A}(x) + \lambda \sigma_{A}(y)$$ 
$$(\tau; \lambda)(x, y) = (1- \lambda) \tau_{A}(x) + \lambda \tau_{A}(y)$$ 
and $$(\gamma; \lambda)(x,y) = (1- \lambda) \gamma_{A}(x) + \lambda \gamma_{A}(y)$$ 
 $\forall~ x, y \in E.$ 
\end{remark}

\begin{definition} \cite{dp}
Let $A  = \lbrace (x, \sigma_{A}, \tau_{A}, \gamma_{A})| x \in X \rbrace$ be PFS over the universe $X$. Then, $(r,s,t)$-cut of $A$ is the crisp set in $A$, denoted by $C_{r, s, t} (A)$ and is defined by $$C_{r, s, t} (A) = \left\lbrace x \in X | \sigma_{A}(x) \geq r,~\tau_{A}(x) \geq s, \gamma_{A}(x) \leq t  \right\rbrace$$ $r, s, t \in [0, 1]$ with the condition  $0 \leq r + s + t \leq 1$
\end{definition}
\ \\
\noindent Brown \cite{b} and Zadeh \cite{z} established that convexity of fuzzy set and cut set of fuzzy set are equivalent.

\begin{lemma} \cite{{b}, {z}}
A fuzzy set $A$ is convex if and only if its cut sets are convex.
\end{lemma} 

\begin{definition} \cite{y}
A multiset or bag $A$ drawn from a set $X$ is characterised by a function $Count_A$ such that $ Count_A: X \rightarrow N $ defined by $ C_A(x) = n \in \mathbb{N}$ where $N$ is the set of non-negative integers.
\end{definition}

\begin{definition}\cite{y}
A fuzzy multiset (FMS) $\mathcal{D}$ drawn from $\mathcal{C}$ is characterised by a count membership function $cm_{\mathcal{D}}$ such that $cm_{\mathcal{D}}: \mathcal{C} \rightarrow \mathcal{N}$, where $\mathcal{N}$ is the set of all crisp multisets drawn from $\mathbb{I}$. Then, for any $r \in \mathcal{C}$, the value $cm_{\mathcal{D}}(r)$ is a crisp multiset drawn from $\mathbb{I}.$ For any $r \in \mathcal{C}$, the membership sequence is defined as the decreasingly ordered sequence of elements in $cm_{\mathcal{D}}(r)$. It is denoted by $(\sigma^{1}_{\mathcal{\mathcal{D}}}(r), \sigma^{2}_{\mathcal{\mathcal{D}}}(r), \cdots, \sigma^{d}_{\mathcal{\mathcal{D}}}(r))$ where $\sigma^{1}_{\mathcal{D}}(r) \geq \sigma^{2}_{\mathcal{D}}(r) \geq \cdots \geq \sigma^{d}_{\mathcal{D}}(r).$
\end{definition}

\begin{definition} \cite{cfs}
Given a nonempty set $\mathcal{C}.$ The picture fuzzy multiset (PFMS) $\mathcal{D}$ in $\mathcal{C}$ is characterised by three functions namely positive membership count function $pmc$, neutral membership count function $n_{e}mc$ and negative membership count function $nmc$ such that $pmc, n_{e}mc, nmc: \mathcal{C} \rightarrow \mathcal{N}$,
where $\mathcal{N},$ is refer to collection of crisp multisets taken from $\mathbb{I}.$ Thus, every element $r \in \mathcal{C}$, $pmc$ is the crisp multiset from $\mathbb{I}$ whose positive membership sequence is defined by $(\sigma^{1}_{\mathcal{D}}(r), \sigma^{2}_{\mathcal{D}}(r), \cdots, \sigma^{n}_{\mathcal{D}}(r))$ such that $\sigma^{1}_{\mathcal{D}}(r) \geq \sigma^{2}_{\mathcal{D}}(r) \geq \cdots \geq \sigma^{n}_{\mathcal{D}}(r)$, $n_{e}mc$ is the crisp multiset from $\mathbb{I}$ whose neutral membership sequence is defined by $(\tau^{1}_{\mathcal{D}}(r), \tau^{2}_{\mathcal{D}}(r), \cdots, \tau^{n}_{\mathcal{D}}(r))$ and $nmc$ is the crisp multiset from $\mathbb{I}$ whose negative membership sequence is defined by $(\eta^{1}_{\mathcal{D}}(r), \eta^{2}_{\mathcal{D}}(r), \cdots, \eta^{n}_{\mathcal{D}}(r))$, these can be either decreasing or increasing functions satisfying $0 \leq \sigma^{k}_{\mathcal{D}}(r) + \tau^{k}_{\mathcal{D}}(r) + \eta^{k}_{\mathcal{D}}(r) \leq 1$ $ \forall r \in \mathcal{C},$~ $k = 1, 2, \cdots, n.$\\ \ \\ Thus, $\mathcal{D}$ is represented by $$\mathcal{D} = \lbrace \langle r, \sigma^{k}_{\mathcal{D}}(r), \tau^{k}_{\mathcal{D}}(r), \eta^{k}_{\mathcal{D}}(r)   \rangle~ |~ r \in \mathcal{C}  \rbrace,$$ $k = 1, 2, \cdots, n.$
\end{definition}

\noindent The set of all picture fuzzy multisets over $\mathcal{C},$ is denoted as PFMS($\mathcal{C}$).

\begin{definition} \cite{s1} \label{1}
Let $A$ be a PFS on $E$. Then, $A$ is said to be convex, if for all $x, y \in E$ and $\lambda \in \mathbb{I}$,
$$ \sigma_{A} [(1 - \lambda)x + \lambda y] \geq \sigma_{A}(x) \wedge \sigma_{A}(y),$$ $$\tau_{A} [(1 - \lambda)x + \lambda y] \geq \tau_{A}(x) \wedge \tau_{A}(y)$$ and $$\gamma_{A}[(1 - \lambda)x + \lambda y] \leq \gamma_{A}(x) \vee \gamma_{A}(y)$$
\end{definition}

\noindent In \cite{s1}, Definition \ref{1} to finitely many points of PFS $A$ as follows:

\begin{definition}  \cite{s1} \label{2}
Let $A$ be a PFS on $E$, $x_1, \cdots, x_{n} \in A$. Then, the picture convex combination of $A$, denoted by $PCC(A)$ is a point $x$ with $x = \sum \limits_{i=1}^{n} \lambda_{i} x_i$ where $$x = (\sigma_x, \tau_x, \gamma_x ),~~~ \lambda_i = (\lambda_{i \sigma}, \lambda_{i \tau}, \lambda_{i \gamma} ),~~~ x_i = ( \sigma_{A}(x_i), \tau_{A}(x_i), \gamma_{A}(x_i) ),$$~~~ $$\sigma_x = \sum \limits_{i=1}^{n} \lambda_{i \sigma} \sigma_{A}(x_i), ~~~~~~~~~~ \tau_x = \sum \limits_{i=1}^{n} \lambda_{i \tau} \tau_{A}(x_i),~~~~~~~~~~ \gamma_x = \sum \limits_{i=1}^{n} \lambda_{i \gamma} \gamma_{A}(x_i)$$  $\sum \limits_{i=1}^{n} (\lambda_{i \sigma} + \lambda_{i \tau} + \lambda_{i \gamma}) = 1$~ and~ $ (\lambda_{i \sigma} + \lambda_{i \tau} + \lambda_{i \gamma}) \in \mathbb{I}$ for each $i$.
\end{definition}

\noindent In \cite{s1}, Definition \ref{1} was defined in terms of Definition \ref{2} as;

\begin{definition}  \cite{s1}
Let $A$ be a PFS on $E$ and $x_1, \cdots, x_{n}$ be points in $A$. Then, $A$ is said to be convex if for all $x_{i} \in A,~ i = 1, \cdots, n$ and $\lambda \in \mathbb{I}$ $$\sigma_A \{ \sum \limits_{i=1}^{n} \lambda_{i} x_i \} \geq \sigma_{A}(x_1) \wedge \cdots \wedge \sigma_{A}(x_n),$$ $$\tau_A \{ \sum \limits_{i=1}^{n} \lambda_{i} a_i \} \geq \tau_{A}(x_1) \wedge \cdots \wedge \tau_{A}(x_n),$$ and $$\gamma_A \{ \sum \limits_{i=1}^{n} \lambda_{i} a_i \} \leq \gamma_{A}(x_1) \vee \cdots \vee \gamma_{A}(x_n)$$ and $\sum \limits_{i=1}^{d} \lambda_i = 1$. Where $\lambda_i$ and $x_i$ are as defined above.
\end{definition}

\begin{definition}\cite{s1}
Let $A = \lbrace (x, \sigma_{A}(x), \tau_{A}(x), \gamma_{A}(x))| x \in E \rbrace$ be PFS. Then, the picture convex hull of $A$, denoted by $Pch(A)$ is defined as $$Pch(A):= \lbrace x = \sum \limits_{i=1}^{n} \lambda_{i} x_i : x_{i} \in A; n \geq 1, \lambda_{i} \in \mathbb{I}, \sum \limits_{i=1}^{n} \lambda_{i} = 1 \rbrace$$ where $x$, $\lambda_i$ and $x_i$ are as defined in Definition \ref{2} and the positive, neutral and negative membership degrees are respectively defined as follow: $$\sigma_{Pch(A)}(x) = \wedge \lbrace \sigma_{B}(x) | \sigma_{B} \geq \sigma_{A}(x)\rbrace,$$ $$\tau_{Pch(A)}(x)(x) = \wedge \lbrace \tau_{B}(x) | \tau_{B} \geq \tau_{A}(x)\rbrace$$ $$\gamma_{Pch(A)}(x)(x) = \vee \lbrace \gamma_{B}(x) | \gamma_{B} \leq \gamma_{A}(x)\rbrace,$$ where $B$ is a PFS of the form $B = \lbrace ( x, \sigma_{B}(x), \tau_{B}(x), \gamma_{B}(x) ) | x \in E \rbrace$.  
\end{definition}

\begin{proposition} \cite{s1}
A PFS $A = \lbrace ( x, \sigma_{A}(x), \tau_{A}(x), \gamma_{A}(x) ) | x \in E \rbrace$ is PCFS if and only if $C_{r, s, t} (A)$ is a PCFS.
\end{proposition}

\begin{proposition} \cite{s1}
Let $A = \lbrace ( x, \sigma_{A}(x), \tau_{A}(x), \gamma_{A}(x) ) | x \in E \rbrace$ and $B = \lbrace ( x, \sigma_{B}(x), \tau_{B}(x), \gamma_{B}(x) ) | x \in E \rbrace$ be PCFSs. Then, $D = A \cap B$ is a PCFS.
\end{proposition}

\begin{corollary} \cite{s1}
Given any arbitrary family $\{ A_i, i = 1, 2, \cdots, \}$ of PCFSs in $E$, then their intersection is also a PCFS.
\end{corollary}

\begin{proposition} \cite{s1}
A PFS $A$ is a PCFS if and only if for all finite family $r_{i} \in E$ and $\lambda_{i} \in \mathbb{I},~ i = 1, 2, \cdots n$ such that $\sum \limits_{i=1}^{n} \lambda_i = 1$, we have $$\sigma_{A} \left(\sum \limits_{i=1}^{n} \lambda_{i} r_{i} \right) \geq \bigwedge^{n}_{i=1} \sigma_{A} (r_i),~~~~\tau_{A} \left(\sum \limits_{i=1}^{n} \lambda_{i} r_{i} \right) \geq \bigwedge^{n}_{i=1} \tau_{A} (r_i), ~~~~\gamma_{A} \left(\sum \limits_{i=1}^{n} \lambda_{i} r_{i} \right) \leq \bigvee^{n}_{i=1} \gamma_{A} (r_i) \ \ \ \ \ \ \ \ \ \ (1).$$ 
\end{proposition}

\begin{proposition} \cite{s1}
For a PCFS $A$ on $E$, $Pch(A)$ consists of all the convex combinations of the points of $A$.
\end{proposition}

\begin{definition}\cite{s1} 
Let $A$ be a PFS on $E$. Then, $A$ is called a picture convex fuzzy set, (PCFS) if for all $x, y \in E$ and $\lambda \in \mathbb{I}$,
$$ \sigma_{A} [(1 - \lambda)x + \lambda y] \geq \sigma_{A}(x) \wedge \sigma_{A}(y),$$ $$\tau_{A} [(1 - \lambda)x + \lambda y] \geq \tau_{A}(x) \wedge \tau_{A}(y)$$ and $$\gamma_{A}[(1 - \lambda)x + \lambda y] \leq \gamma_{A}(x) \vee \gamma_{A}(y)$$
\end{definition}

\section{Notions on Convexity of Picture Fuzzy Multisets}
\begin{definition} \label{f}
Let $\mathcal{D}$ be a PFMS on $\mathcal{C}$. Then, $\mathcal{D}$ is said to be convex, if for all $x, y \in \mathcal{C}$ and $\lambda \in \mathbb{I}$,
$$ \sigma^k_{\mathcal{D}} [(1 - \lambda)x + \lambda y] \geq \sigma^k_{\mathcal{D}}(x) \wedge \sigma^k_{\mathcal{D}}(y),$$ $$\tau^k_{\mathcal{D}} [(1 - \lambda)x + \lambda y] \geq \tau^k_{\mathcal{D}}(x) \wedge \tau^k_{\mathcal{D}}(y)$$ and $$\eta^k_{\mathcal{D}}[(1 - \lambda)x + \lambda y] \leq \eta^k_{\mathcal{D}}(x) \vee \eta^k_{\mathcal{D}}(y)$$
\end{definition}

\noindent Remark \ref{i} can be extended to finitely many points of PFMS $\mathcal{D}$ as follows:

\begin{definition}  \label{g}
Let $\mathcal{D}$ be a PFMS on $\mathcal{C},$ $x_1, \cdots, x_{n} \in \mathcal{D}$. Then, the picture convex combination of $\mathcal{D},$ denoted by $PCC(\mathcal{D})$ is a point $x$ with $x = \sum \limits_{i=1}^{n} \lambda_{i} x_i$ where $$x = (\sigma^k_x, \tau^k_x, \eta^k_x ),~~~ \lambda_i = (\lambda_{i \sigma^k}, \lambda_{i \tau^k}, \lambda_{i \eta^k} ),~~~ x_i = ( \sigma^k_{\mathcal{D}}(x_i), \tau^k_{\mathcal{D}}(x_i), \eta^k_{\mathcal{D}}(x_i) ),$$~~~ $$\sigma^k_x = \sum \limits_{i=1}^{n} \lambda_{i \sigma^k} \sigma^k_{\mathcal{D}}(x_i), ~~~~~~~~~~ \tau^k_x = \sum \limits_{i=1}^{n} \lambda_{i \tau^k} \tau^k_{\mathcal{D}}(x_i),~~~~~~~~~~ \eta^k_x = \sum \limits_{i=1}^{n} \lambda_{i \eta^k} \eta^k_{\mathcal{D}}(x_i)$$  $\sum \limits_{i=1}^{n} (\lambda_{i \sigma^k} + \lambda_{i \tau^k} + \lambda_{i \eta^k}) = 1$~ and~ $ (\lambda_{i \sigma^k} + \lambda_{i \tau^k} + \lambda_{i \eta^k}) \in \mathbb{I}$ for each $i$.
\end{definition}

\noindent Definition \ref{f} can be defined in terms of Definition \ref{g} as;

\begin{definition} 
Let $\mathcal{D}$ be a PFMS on $\mathcal{C}$ and $x_1, \cdots, x_{n}$ be points in $\mathcal{D}$. Then, $\mathcal{D}$ is said to be convex if for all $x_{i} \in \mathcal{D},~ i = 1, \cdots, n$ and $\lambda \in \mathbb{I}$ $$\sigma^k_\mathcal{D} \{ \sum \limits_{i=1}^{n} \lambda_{i} x_i \} \geq \sigma^k_{\mathcal{D}}(x_1) \wedge \cdots \wedge \sigma^k_{\mathcal{D}}(x_n),$$ $$\tau^k_\mathcal{D} \{ \sum \limits_{i=1}^{n} \lambda_{i} a_i \} \geq \tau^k_{\mathcal{D}}(x_1) \wedge \cdots \wedge \tau^k_{\mathcal{D}}(x_n),$$ and $$\eta^k_\mathcal{D} \{ \sum \limits_{i=1}^{n} \lambda_{i} a_i \} \leq \eta^k_{\mathcal{D}}(x_1) \vee \cdots \vee \eta^k_{\mathcal{D}}(x_n)$$ and $\sum \limits_{i=1}^{d} \lambda_i = 1$. Where $\lambda_i$ and $x_i$ are as defined above.
\end{definition}

\begin{definition} 
Let $$\mathcal{D} = \lbrace \langle r, \sigma^{k}_{\mathcal{D}}(r), \tau^{k}_{\mathcal{D}}(r), \eta^{k}_{\mathcal{D}}(r)   \rangle~ |~ r \in \mathcal{C}  \rbrace,$$ $k = 1, 2, \cdots, n.$ be PFMS. Then, the picture convex hull of $\mathcal{D}$, denoted by $Pch(\mathcal{D})$ is defined as $$Pch(\mathcal{D}):= \lbrace x = \sum \limits_{i=1}^{n} \lambda_{i} x_i : x_{i} \in \mathcal{D}; n \geq 1, \lambda_{i} \in \mathbb{I}, \sum \limits_{i=1}^{n} \lambda_{i} = 1 \rbrace$$ where $x$, $\lambda_i$ and $x_i$ are as defined in Definition \ref{2} and the positive, neutral and negative membership degrees are respectively defined as follow: $$\sigma^k_{Pch(\mathcal{D})}(x) = \wedge \lbrace \sigma^k_{\mathcal{E}}(x) | \sigma^k_{\mathcal{E}} \geq \sigma^k_{\mathcal{D}}(x)\rbrace,$$ $$\tau^k_{Pch(\mathcal{D})}(x)(x) = \wedge \lbrace \tau^k_{\mathcal{E}}(x) | \tau^k_{\mathcal{E}} \geq \tau^k_{\mathcal{D}}(x)\rbrace$$ $$\eta^k_{Pch(\mathcal{D})}(x)(x) = \vee \lbrace \eta^k_{\mathcal{E}}(x) | \eta^k_{\mathcal{E}} \leq \eta^k_{\mathcal{D}}(x)\rbrace,$$ where $\mathcal{\mathcal{E}}$ is a PFMS of the form $$\mathcal{E} = \lbrace \langle r, \sigma^{k}_{\mathcal{E}}(r), \tau^{k}_{\mathcal{E}}(r), \eta^{k}_{\mathcal{E}}(r)   \rangle~ |~ r \in \mathcal{C}  \rbrace,$$ $k = 1, 2, \cdots, n.$  
\end{definition}

\begin{proposition} \label{h} 
Let $$\mathcal{D} = \lbrace \langle r, \sigma^{k}_{\mathcal{D}}(r), \tau^{k}_{\mathcal{D}}(r), \eta^{k}_{\mathcal{D}}(r)   \rangle~ |~ r \in \mathcal{C}  \rbrace,$$ $k = 1, 2, \cdots, n.$ be PFMS. Then, $\mathcal{D}$ is convex if and only if $C_{r, s, t} (\mathcal{D})$ is convex.
\end{proposition}

\begin{proof} Suppose that $\mathcal{D}$ is convex. Let $x, y \in  C_{r, s, t} (\mathcal{D})$, then $\sigma^k_{\mathcal{D}} (x) \geq r$, $\sigma^k_{\mathcal{D}} (y) \geq r,$ $\tau^k_{\mathcal{D}} (x) \geq s$, $\tau^k_{\mathcal{D}} (y) \geq s$ and $\eta^k_{\mathcal{D}} (x) \leq t$, $\eta^k_{\mathcal{D}} (y) \leq t$ thus $\sigma^k_{\mathcal{D}}(x) \wedge \sigma^k_{\mathcal{D}}(y) \geq r,$ $\tau^k_{\mathcal{D}}(x) \wedge \tau^k_{\mathcal{D}}(y) \geq s$ and $\eta^k_{\mathcal{D}}(x) \vee \eta^k_{\mathcal{D}}(y) \leq t$. Since $\mathcal{D}$ is convex, then $$\sigma^k_{\mathcal{D}} [(1 - \lambda)x + \lambda y] \geq \sigma^k_{\mathcal{D}}(x) \wedge \sigma^k_{\mathcal{D}}(y) \geq r,$$ $$\tau^k_{\mathcal{D}} [(1 - \lambda)x + \lambda y] \geq \tau^k_{\mathcal{D}}(x) \wedge \tau^k_{\mathcal{D}}(y) \geq s,$$ $$\eta^k_{\mathcal{D}} [(1 - \lambda)x + \lambda y] \leq \tau^k_{\mathcal{D}}(x) \vee \tau^k_{\mathcal{D}}(y) \leq t$$ for all $\lambda \in \mathbb{I}$. Thus, $(1 - \lambda)x + \lambda y \in C_{r, s, t} (\mathcal{D})$. Hence, $C_{r, s, t} (\mathcal{D})$ is convex.\\ Conversely, suppose that $C_{r, s, t} (\mathcal{D})$ is convex for all $r, s, t \in \mathbb{I}$, let $x, y \in \mathcal{C},~ \lambda \in \mathbb{I}$. Let $ r =\sigma^k_{\mathcal{D}}(x) \wedge \sigma^k_{\mathcal{D}}(y),$ $s = \tau^k_{\mathcal{D}}(x) \wedge \tau^k_{\mathcal{D}}(y)$, and $t = \eta^k_{\mathcal{D}}(x) \vee \eta^k_{\mathcal{D}}(y)$. Now $$\sigma^k_{\mathcal{D}}(x) \geq \sigma^k_{\mathcal{D}}(x) \wedge \sigma^k_{\mathcal{D}}(y) = r,~ \sigma^k_{\mathcal{D}}(y) \geq \sigma^k_{\mathcal{D}}(x) \wedge \sigma^k_{\mathcal{D}}(y) = r,$$ $$\tau^k_{\mathcal{D}}(x) \geq \tau^k_{\mathcal{D}}(y) \wedge \tau^k_{\mathcal{D}}(y) = s,~ \tau^k_{\mathcal{D}}(y) \geq \tau^k_{\mathcal{D}}(x) \wedge \tau^k_{\mathcal{D}}(y) = s,$$ $$\eta^k_{\mathcal{D}}(x) \leq \eta^k_{\mathcal{D}}(y) \vee \eta^k_{\mathcal{D}}(y) = t,~ \eta^k_{\mathcal{D}}(y) \leq \eta^k_{\mathcal{D}}(x) \vee \eta^k_{\mathcal{D}}(y) = t$$ which imply that $x, y \in C_{r, s, t} (\mathcal{D})$. Since $C_{r, s, t} (\mathcal{D})$ is convex, we have $(1 - \lambda)x + \lambda y \in C_{r, s, t} (\mathcal{D}).$ Thus, $$\sigma^k_{\mathcal{D}} [(1 - \lambda)x + \lambda y] \geq r = \sigma^k_{\mathcal{D}}(x) \wedge \sigma^k_{\mathcal{D}}(y),$$ $$\tau^k_{\mathcal{D}} [(1 - \lambda)x + \lambda y] \geq s = \tau^k_{\mathcal{D}}(x) \wedge \tau^k_{\mathcal{D}}(y)$$ and $$\eta^k_{\mathcal{D}} [(1 - \lambda)x + \lambda y] \leq t = \eta^k_{\mathcal{D}}(x) \vee \eta^k_{\mathcal{D}}(y)$$ Therefore, $\mathcal{D}$ is a convex.
\end{proof}

\begin{proposition} \label{i}
Let $\mathcal{D}$ and $\mathcal{E}$ be in PFMS($\mathcal{C}$) and convex. Then, $\mathcal{A} = \mathcal{D} \cap \mathcal{E}$ is a convex.
\end{proposition}

\begin{proof}
Let $x, y \in \mathcal{C}$, take $\lambda \in \mathbb{I}$. Then, $$\sigma^k_{\mathcal{A}} [(1 - \lambda)x + \lambda y] = \sigma^k_{\mathcal{D}} [(1 - \lambda)x + \lambda y] \cap \sigma^k_{\mathcal{E}} [(1 - \lambda)x + \lambda y],$$ $$\tau^k_{\mathcal{A}} [(1 - \lambda)x + \lambda y] = \tau^k_{\mathcal{D}} [(1 - \lambda)x + \lambda y] \cap \tau^k_{\mathcal{E}} [(1 - \lambda)x + \lambda y]$$ and $$\eta^k_{\mathcal{A}} [(1 - \lambda)x + \lambda y] = \eta^k_{\mathcal{D}} [(1 - \lambda)x + \lambda y] \cup \eta^k_{\mathcal{E}} [(1 - \lambda)x + \lambda y]$$ Since $\mathcal{D}$ and $\mathcal{E}$ are convex, $$\sigma^k_{\mathcal{A}} [(1 - \lambda)x + \lambda y] \geq \sigma^k_{\mathcal{D}}(x) \wedge \sigma^k_{\mathcal{D}}(y),~ \sigma^k_{\mathcal{E}} [(1 - \lambda)x + \lambda y] \geq \sigma^k_{\mathcal{E}}(x) \wedge \sigma^k_{\mathcal{E}}(y),$$ $$\tau^k_{\mathcal{A}} [(1 - \lambda)x + \lambda y] \geq \tau^k_{\mathcal{D}}(x) \wedge \tau^k_{\mathcal{D}}(y),~ \tau^k_{\mathcal{E}} [(1 - \lambda)x + \lambda y] \geq \tau^k_{\mathcal{E}}(x) \wedge \tau^k_{\mathcal{E}}(y)$$ and $$\eta^k_{\mathcal{A}} [(1 - \lambda)x + \lambda y] \leq \eta^k_{\mathcal{D}}(x) \vee \eta^k_{\mathcal{D}}(y),~ \eta^k_{\mathcal{E}} [(1 - \lambda)x + \lambda y] \leq \eta^k_{\mathcal{E}}(x) \vee \eta^k_{b}(y).$$ Thus, 
\begin{eqnarray*} \sigma^k_{\mathcal{A}} [(1 - \lambda)x + \lambda y]  
& = & [\sigma^k_{\mathcal{D}} (x) \wedge \sigma^k_{\mathcal{D}} (y)] \cap [\sigma^k_{\mathcal{E}} (x) \wedge \sigma^k_{\mathcal{E}} (y)]\\
&\geq & [\sigma^k_{\mathcal{D}}(x) \cap \sigma^k_{\mathcal{E}}(x)] \wedge [\sigma^k_{\mathcal{D}}(y) \cap \sigma^k_{\mathcal{E}}(y)]\\
&\geq &[\sigma^k_{\mathcal{D}} \cap \sigma^k_{\mathcal{E}}](x) \wedge [\sigma^k_{\mathcal{D}} \cap \sigma^k_{\mathcal{E}}](y)\\ & \geq & [\sigma^k_{\mathcal{D}\cap \mathcal{E}}](x) \wedge [\sigma^k_{\mathcal{D} \cap \mathcal{E} }](y)\\ & = & \sigma^k_{\mathcal{A}} (x) \wedge \sigma^k_{\mathcal{A}} (y).
\end{eqnarray*},
\begin{eqnarray*} \tau^k_{\mathcal{A}} [\lambda x + ((1 - \lambda)x + \lambda y]  
& = & [\tau^k_{\mathcal{D}} (x) \wedge \tau^k_{\mathcal{D}} (y)] \cap [\tau^k_{\mathcal{E}} (x) \wedge \tau^k_{\mathcal{E}} (y)]\\
&\geq & [\tau^k_{\mathcal{D}}(x) \cap \tau^k_{\mathcal{E}}(x)] \wedge [\tau^k_{\mathcal{D}}(y) \cap \tau^k_{\mathcal{E}}(y)]\\
&\geq &[\tau^k_{\mathcal{D}} \cap \tau^k_{\mathcal{E}}](x) \wedge [\tau^k_{\mathcal{D}} \cap \tau^k_{\mathcal{E}}](y)\\ & \geq & [\tau^k_{\mathcal{D} \cap \mathcal{E}}](x) \wedge [\tau^k_{\mathcal{D}  \cap \mathcal{E}}](y)\\ & = & \tau^k_{\mathcal{A}} (x) \wedge \tau^k_{\mathcal{A}} (y).
\end{eqnarray*}
and 
\begin{eqnarray*} \eta^k_{\mathcal{A}} [(1 - \lambda)x + \lambda y]  
& = & [\eta^k_{\mathcal{D}} (x) \vee \eta^k_{\mathcal{D} } (y)] \cup [\eta^k_{\mathcal{E}} (x) \vee \eta^k_{\mathcal{E}} (y)]\\
&\leq & [\eta_{\mathcal{D}}(x) \cup \eta^k_{\mathcal{E} }(x)] \vee [\eta^k_{\mathcal{D}}(y) \cup \eta^k_{\mathcal{E}}(y)]\\
&\leq &[\eta^k_{\mathcal{D}} \cup \eta^k_{\mathcal{E}}](x) \vee [\eta^k_{\mathcal{D}} \cup \eta^k_{\mathcal{E}}](y)\\ & \leq & [\eta^k_{\mathcal{D} \cup \mathcal{E}}](x) \vee [\eta^k_{\mathcal{D}  \cup \mathcal{E}}](y)\\ & = & \eta^k_{\mathcal{A}} (x) \vee \eta^k_{\mathcal{A}} (y).
\end{eqnarray*}
\end{proof}

\begin{corollary} \label{j} 
Given any arbitrary family $\{ \mathcal{D}_i, i = 1, 2, \cdots, \}$ of PCFSs in $\mathcal{C}$, then their intersection is also convex.
\end{corollary}

\begin{proof}
Let $\lbrace \mathcal{D}_{i} \rbrace_{i}$ be an arbitrary family of PCFS of $\mathcal{C}$. Then, the intersection $\mathcal{G} = \bigwedge_{i} \mathcal{D}_i$ is convex. If $$\sigma^k_{\mathcal{G}}[(1 - \lambda)x + \lambda y] \in \mathcal{G},~ \tau^k_{\mathcal{G}}[(1 - \lambda)x + \lambda y] \in \mathcal{G}~ \text{and}~ \eta^k_{\mathcal{G}}[(1 - \lambda)x + \lambda y] \in \mathcal{G},$$ then they belong also to every $\mathcal{D}_i,$ since $\mathcal{D}_i$ is convex. It implies that $$\sigma^k_{\mathcal{G}}[(1 - \lambda)x + \lambda y] \in \mathcal{G},~ \tau^k_{\mathcal{G}}[(1 - \lambda)x + \lambda y] \in \mathcal{G}~ \text{and}~ \eta^k_{\mathcal{G}}[(1 - \lambda)x + \lambda y] \in \mathcal{G}$$ belong to every $\mathcal{D}_i,$ and consequently to their intersection, i.e., to $\mathcal{G}$. 
\end{proof}
\ \\
\begin{proposition} \label{k} 
A PFMS $\mathcal{D}$ is a convex if and only if for all finite family $r_{i} \in \mathcal{C}$ and $\lambda_{i} \in \mathbb{I},~ i = 1, 2, \cdots n$ such that $\sum \limits_{i=1}^{n} \lambda_i = 1$, we have $$\sigma^k_{\mathcal{D}} \left(\sum \limits_{i=1}^{n} \lambda_{i} r_{i} \right) \geq \bigwedge^{n}_{i=1} \sigma^k_{\mathcal{D}} (r_i),~~~~\tau^k_{\mathcal{D}} \left(\sum \limits_{i=1}^{n} \lambda_{i} r_{i} \right) \geq \bigwedge^{n}_{i=1} \tau^k_{\mathcal{D}} (r_i), ~~~~\eta^k_{\mathcal{D}} \left(\sum \limits_{i=1}^{n} \lambda_{i} r_{i} \right) \leq \bigvee^{n}_{i=1} \eta^k_{\mathcal{D}} (r_i) \ \ \ \ \ \ \ \ \ \ (1).$$ 
\end{proposition}
\ \\
\begin{proof}
Assume that $\mathcal{D}$ contains all finite family $r_{i} \in \mathcal{C}$ and $\lambda_{i} \in \mathbb{I},~ i = 1, 2, \cdots n$ such that $\sum \limits_{i=1}^{n} \lambda_i = 1$, then this holds for two points $a,~ b \in \mathcal{C}$ where $a = \sum \limits_{i=1}^{p} \lambda_{i} r_i,~~ a = ( \sigma^k_{a}, \tau^k_{a}, \eta^k_{a} )$,\\ $\lambda_i = ( \lambda_{i \sigma^k_1}, \lambda_{i \tau^k_1}, \lambda_{i \eta^k_1} )$ and $r_i = (\sigma^k_{\mathcal{D}}(r_i), \tau^k_{\mathcal{D}}(r_i), \eta^k_{\mathcal{D}}(r_i) )$, $\sigma^k_{a} = \sum \limits_{i=1}^{p} \lambda_{i \sigma^k_{1}} \sigma^k_{\mathcal{D}}(r_i)$,~ $\tau^k_{a} = \sum \limits_{i=1}^{p} \lambda_{i \tau^k_{1}} \tau^k_{\mathcal{D}}(r_i)$,\\ $\eta^k_{a} = \sum \limits_{i=1}^{p} \lambda_{i \eta^k_{1}} \eta^k_{\mathcal{D}}(r_i)$, $~~\sum \limits_{i=1}^{p} (\lambda_{i \sigma^k_{1}} + \lambda_{i \tau^k_{1}} + \lambda_{i \eta^k_{1}}) = 1$~ and~ $(\lambda_{i \sigma^k_{1}} + \lambda_{i \tau^k_{1}} + \lambda_{i \eta^k_{1}}) \in \mathbb{I}$ for all $i$ and \\$b = \sum \limits_{i=1}^{p} \phi_{i} r_i,~~ b = ( \sigma^k_{b}, \tau^k_{b}, \eta^k_{b} )$,~~ $\phi_i = ( \phi_{i \sigma^k_{2}}, \phi_{i \tau^k_{2}}, \phi_{i \eta^k_{2}} )$ and $r_i = ( \sigma^k_{\mathcal{D}}(r_i), \tau^k_{\mathcal{D}}(r_i), \eta^k_{\mathcal{D}})$,\\ $\sigma^k_{b} = \sum \limits_{i=1}^{p} \phi_{i \sigma^k_{2}} \sigma^k_{\mathcal{D}}(r_i)$, $\tau^k_{b} = \sum \limits_{i=1}^{p} \phi_{i \tau^k_{2}} \tau^k_{\mathcal{D}}(r_i)$, $\eta^k_{b} = \sum \limits_{i=1}^{p} \phi_{i \eta_{2}} \eta^k_{\mathcal{D}}(r_i)$, $~~\sum \limits_{i=1}^{p} (\phi_{i \sigma^k_{2}} + \phi_{i \tau^k_{2}} + \phi_{i \eta^k_{2}}) = 1$~ and\\ $(\phi_{i \sigma^k_{2}} + \phi_{i \tau^k_{2}} + \phi_{i \eta^k_{2}}) \in \mathbb{I}$ for all $i$.\\ Thus, $$\sigma^k_{\mathcal{D}} \left( ( 1 - \delta) a + \delta b \right) \geq \sigma^k_{\mathcal{D}}(a) \wedge \sigma^k_{\mathcal{D}}(b),$$ $$\tau^k_{\mathcal{D}} \left( ( 1 - \delta) a + \delta b \right) \geq \tau^k_{\mathcal{D}}(a) \wedge \tau^k_{\mathcal{D}}(b)$$ and $$\eta^k_{\mathcal{D}} \left( ( 1 - \delta) a + \delta b \right) \leq \eta^k_{\mathcal{D}}(a) \vee \eta^k_{\mathcal{D}}(b)$$ The positive, neutral and negative membership degrees of $a$ and $b$ are given by
\begin{eqnarray*}
\sigma^k_{\mathcal{D}}& = & (1 - \delta) \sum \limits_{i=1}^{p} \lambda_{i \sigma_{1}}\sigma^k_{\mathcal{D}}(r_i) +  \delta \sum \limits_{i=1}^{p} \phi_{i \sigma^k_{2}} \sigma^k_{\mathcal{D}}(r_i)\\
& = & \sum \limits_{i=1}^{p}\left( (1 - \delta) \lambda_{i \sigma^k_{1}} +  \delta \phi_{i \sigma^k_{2}} \right) \sigma^k_{\mathcal{D}}(r_i)
\end{eqnarray*}
	
\begin{eqnarray*}
\tau^k_{\mathcal{D}}& = & (1 - \delta) \sum \limits_{i=1}^{p} \lambda_{i \tau^k_{1}} \tau^k_{\mathcal{D}}(r_i) + \delta \sum \limits_{i=1}^{p} \phi_{i \tau^k_{2}} \tau^k_{\mathcal{D}}(r_i)\\
& = & \sum \limits_{i=1}^{p}\left( (1 - \delta) \lambda_{i \tau^k_{1}} +  \delta \phi_{i \tau^k_{2}} \right) \tau^k_{\mathcal{D}}(r_i)
\end{eqnarray*}  and 
\begin{eqnarray*}
\eta^k_{\mathcal{D}}& = & (1 - \delta) \sum \limits_{i=1}^{p} \lambda_{i \eta_{1}} \eta^k_{\mathcal{D}}(r_i) + \delta \sum \limits_{i=1}^{p} \phi_{i \eta_{2}} \eta^k_{\mathcal{D}}(r_i)\\
& = & \sum \limits_{i=1}^{p}\left( (1 - \delta) \lambda_{i \eta^k_{1}} +  \delta \phi_{i \eta^k_{2}} \right) \eta^k_{\mathcal{D}}(r_i)
\end{eqnarray*}
respectively with 
$$\sum \limits_{i=1}^{p}\left[ (1 - \delta) \lambda_{i \sigma^k_{1}} + \delta \phi_{i \sigma^k_{2}} \right] + \sum \limits_{i=1}^{p}\left[ (1 - \delta) \lambda_{i \tau^k_{1}} +  \delta \phi_{i \tau^k_{2}} \right] + \sum \limits_{i=1}^{p}\left[ (1 - \delta) \lambda_{i \eta^k_{1}} +  \delta \phi_{i \eta^k_{2}} \right] $$
\begin{eqnarray*} 
& = & \left(\sum \limits_{i=1}^{p} (1 - \delta) \lambda_{i \sigma^k_{1}} + \sum \limits_{i=1}^{p} (1 - \delta) \lambda_{i \tau^k_{1}} \lambda_{i \tau^k_{1}} + \sum \limits_{i=1}^{p} (1 - \delta) \lambda_{i \eta^k_{1}} \lambda_{i \eta^k_{1}} \right) + \left( \sum \limits_{i=1}^{p} \delta \phi_{i \sigma^k_{2}} + \sum \limits_{i=1}^{p} \delta \phi_{i \tau^k_{2}} + \sum \limits_{i=1}^{p} \delta \phi_{i \eta^k_{2}} \right)\\
& = & (1 - \delta) \left(\sum \limits_{i=1}^{p} \lambda_{i \sigma^k_{1}} + \sum \limits_{i=1}^{p} \lambda_{i \tau^k_{1}} + \sum \limits_{i=1}^{p} \lambda_{i \eta_{1}} \right) + \delta \left( \sum \limits_{i=1}^{p} \phi_{i \sigma^k_{2}} + \sum \limits_{i=1}^{p} \phi_{i \tau^k_{2}} + \sum \limits_{i=1}^{p} \phi_{i \eta^k_{2}} \right)\\
& = & 1 - \delta + \delta\\
&= & 1.
\end{eqnarray*} and since $$(\lambda_{i \sigma^k} + \lambda_{i \tau^k} + \lambda_{i \eta^k}) \in \mathbb{I},~~ (\phi_{i \sigma^k} + \phi_{i \tau^k} + \phi_{i \eta^k}) \in \mathbb{I}$$ then $$(((1 - \delta) \lambda_{i \sigma^k_{1}} + \delta) \phi_{i \sigma^k_{2}}) + ((1 - \delta) \lambda_{i \tau^k_{1}} + \delta \phi_{i \tau^k_{2}})) + ((1 - \delta) \lambda_{i \eta^k_{1}} + \delta \phi_{i \eta^k_{2}}) \in \mathbb{I}$$
such that $$\sigma^k_{\mathcal{D}} [(1 - \delta) a + \lambda b] \geq \sigma^k_{\mathcal{D}}(a) \wedge \sigma^k_{\mathcal{D}}(b),~$$ $$\tau^k_{\mathcal{D}} [(1 - \lambda) a + \lambda b] \geq \tau^k_{\mathcal{D}}(a) \wedge \tau^k_{\mathcal{D}}(b)$$ and $$\eta^k_{\mathcal{D}} [(1 - \lambda) a + \lambda b] \leq \eta^k_{\mathcal{D}}(a) \vee \eta^k_{\mathcal{D}}(b).$$ The convexity of $\mathcal{D}$ means that $\mathcal{D}$ is closed under convex combinations of two points. Therefore, $\mathcal{D}$ is convex.\\ Conversely, suppose that $\mathcal{D}$ is PFMS and $n \in \mathbb{N}$. We shall use induction on $n$. For $n = 1$, the assertion is trivial. For $(n - 1)~ \text{to}~ n,~ n \geq 2$, let $r_1, \cdots, r_n \in \mathcal{D}$ and $\lambda_{i} \in \mathbb{I},~ i = 1, 2, \cdots n$ with $\sum \limits_{i=1}^{n} \lambda_i = 1$. Assume $\lambda_{i} \in \{0, 1 \}$ and $\delta_{i}:=\dfrac{\lambda_{i}}{1 - \lambda_{n}},~ i = 1, \cdots, n-1$, thus $\delta_i \in \mathbb{I}$ and $\sum \limits_{i=1}^{n} \delta_i = 1$.
By the induction hypothesis, $\sum \limits_{i=1}^{n-1} \delta_{i} r_{i} \in \mathcal{D}$. Put $$\sum \limits_{i=1}^{n} \lambda_{i} r_{i} = (1 - \lambda_{n}) \sum \limits_{i=1}^{n-1} \delta_{i} r_{i} + \lambda_{n} r_{n} $$ Then, we have
\begin{eqnarray*}
\sigma^k_{\mathcal{D}} \left(\sum \limits_{i=1}^{n} \lambda_{i} r_{i} \right)
& \geq & \sigma^k_{\mathcal{D}}(\sum \limits_{i=1}^{n-1}\delta_{i} r_{i}) \wedge \sigma^k_{\mathcal{D}}(r_n)\\
& \geq & \sigma^k_{\mathcal{D}}(\sum \limits_{i=1}^{n-1} ( \dfrac{\lambda_{i}}{1 - \lambda_{n}} )(r_{i})) \wedge \sigma^k_{\mathcal{D}}(r_n)\\
& \geq & \bigwedge_{i}^{n} \sigma^k_{\mathcal{D}} (r_i),
\end{eqnarray*}
\begin{eqnarray*}
\tau^k_{\mathcal{D}} \left(\sum \limits_{i=1}^{n} \lambda_{i} r_{i} \right) 
& \geq & \tau^k_{\mathcal{D}}(\sum \limits_{i=1}^{n-1}\delta_{i} r_{i}) \wedge \tau^k_{\mathcal{D}}(r_n)\\
& \geq & \tau^k_{\mathcal{D}}(\sum \limits_{i=1}^{n-1} ( \dfrac{\lambda_{i}}{1 - \lambda_{n}} )(r_{i})) \wedge \tau^k_{\mathcal{D}}(r_n)\\
& \geq & \bigwedge_{i}^{n} \tau^k_{\mathcal{D}} (r_i)
\end{eqnarray*} and 
\begin{eqnarray*}
\eta^k_{\mathcal{D}} \left(\sum \limits_{i=1}^{n} \lambda_{i} r_{i} \right)
& \leq & \eta^k_{\mathcal{D}}(\sum \limits_{i=1}^{n-1}\delta_{i} r_{i}) \vee \eta^k_{\mathcal{D}}(r_n)\\
& \leq & \eta^k_{\mathcal{D}}(\sum \limits_{i=1}^{n-1} ( \dfrac{\lambda_{i}}{1 - \lambda_{n}} )(r_{i})) \vee \eta^k_{\mathcal{D}}(r_n)\\
& \leq & \bigvee_{i}^{n} \eta^k_{\mathcal{D}} (r_i)
\end{eqnarray*} 
\end{proof}
\ \\

\begin{theorem} 
For a PCFS $\mathcal{D}$ on $\mathcal{C}$, $Pch(\mathcal{D})$ consists of all the convex combinations of the points of $\mathcal{D}$.
\end{theorem}

\begin{proof}
The points of $\mathcal{D}$ belong to $Pch(\mathcal{D})$, so all their convex combinations belong to $Pch(\mathcal{D})$ by Proposition \ref{k}. Conversely, let $a,~ b \in \mathcal{C}$ such that $a = \sum \limits_{i=1}^{m} \lambda_{i} r_i,~~ a = ( \sigma^k_{a}, \tau^k_{a}, \eta^k_{a} )$,\\ $\lambda_i = ( \lambda_{i \sigma^k_1}, \lambda_{i \tau^k_1}, \lambda_{i \eta^k_1} )$ and $r_i =  (\sigma^k_{\mathcal{D}}(r_i), \tau^k_{\mathcal{D}}(r_i), \eta^k_{\mathcal{D}}(r_i))$, $\sigma^k_{a} = \sum \limits_{i=1}^{m} \lambda_{i \sigma^k_{1}} \sigma^k_{A}(r_i)$,~ $\tau^k_{a} = \sum \limits_{i=1}^{m} \lambda_{i \tau^k_{1}} \tau^k_{\mathcal{D}}(r_i)$,\\ $\eta^k_{a} = \sum \limits_{i=1}^{m} \lambda_{i \eta^k_{1}} \eta^k_{\mathcal{D}}(r_i)$, $~~\sum \limits_{i=1}^{m} (\lambda_{i \sigma^k_{1}} + \lambda_{i \tau^k_{1}} + \lambda_{i \eta^k_{1}}) = 1$~ and~ $(\lambda_{i \sigma^k_{1}} + \lambda_{i \tau^k_{1}} + \lambda_{i \eta^k_{1}}) \in \mathbb{I}$ for all $i~ \text{and}~ r_{i}~ \in \mathcal{D}$ and $b = \sum \limits_{j=1}^{n} \beta_{j} s_j,~~ b = ( \sigma^k_{b}, \tau^k_{b}, \eta^k_{b} )$,~~ $\beta_j = ( \beta_{j \sigma^k_{2}}, \beta_{j \tau^k_{2}}, \beta_{j \eta^k_{2}} )$ and $s_j = ( \eta^k_{\mathcal{D}}(s_j), \tau^k_{\mathcal{D}}(s_j), \eta_{\mathcal{D}}(s_j) )$,\\ $\sigma^k_{b} = \sum \limits_{j=1}^{n} \beta_{j \sigma^k_{2}} \sigma^k_{\mathcal{D}}(s_j)$, $\tau^k_{b_{2}} = \sum \limits_{j=1}^{n} \beta_{j \tau^k_{2}} \tau_{\mathcal{D}}(s_j)$,~ $\eta^k_{b_{2}} = \sum \limits_{j=1}^{n} \beta_{j \eta^k_{2}} \eta^k_{\mathcal{D}}(s_j)$, $~~~\sum \limits_{j=1}^{n} (\beta_{j \sigma^k_{2}} + \beta^k_{j \tau^k_{2}} +  \beta_{j \eta^k_{2}}) = 1$~ and\\ $(\phi_{i \sigma^k_{2}} + \phi_{i \tau^k_{2}} + \phi_{i \eta^k_{2}}) \in \mathbb{I}$ for all $i~ \text{and}~ s_{i}~ \in \mathcal{D}.$ Thus, 
$$((1 - \delta) \lambda_{i \sigma^k_{1}} + \delta \phi_{j \sigma^k_{2}}) + ((1 - \delta) \lambda_{i \tau^k_{1}} + \delta \phi_{j \tau^k_{2}}) + ((1 - \delta) \lambda_{i \eta^k_{1}} + \delta \phi_{j \eta^k_{2}})) \in \mathbb{I}$$ Now, 
\begin{eqnarray*}
\sigma^k_{\mathcal{D}}(a) \wedge \sigma^k_{\mathcal{D}}(b) & \geq & \sigma^k_{\mathcal{D}}(\sum \limits_{i=1}^{m} \lambda_{i} r_{i}) \wedge \sigma^k_{\mathcal{D}}(\sum \limits_{j=1}^{n} \phi_{j} s_{j})\\
& \geq & \sigma^k_{\mathcal{D}}(r_1) \wedge \cdots \wedge \sigma^k_{\mathcal{D}}(r_m)\wedge \sigma^k_{\mathcal{D}}(s_{1)} \wedge \cdots \wedge \sigma^k_{\mathcal{D}}(s_n)
\end{eqnarray*} 
\begin{eqnarray*}
\tau^k_{\mathcal{D}}(a) \wedge \tau^k_{\mathcal{D}}(b) & \geq & \tau^k_{\mathcal{D}}(\sum \limits_{i=1}^{m} \lambda_{i} r_{i}) \wedge \tau^k_{\mathcal{D}}(\sum \limits_{j=1}^{n} \phi_{j} a_{j})\\
& \geq & \tau^k_{\mathcal{D}}(r_1) \wedge \cdots \wedge \tau^k_{\mathcal{D}}(r_m)\wedge \tau^k_{\mathcal{D}}(s_{1)} \wedge \cdots \wedge \tau^k_{\mathcal{D}}(s_n)
\end{eqnarray*}   
and \begin{eqnarray*}
\eta^k_{\mathcal{D}}(a) \vee \eta^k_{\mathcal{D}}(b) & \leq & \eta^k_{\mathcal{D}}(\sum \limits_{i=1}^{m} \lambda_{i} r_{i}) \vee \eta_{\mathcal{D}}(\sum \limits_{j=1}^{n} \phi_{j} a_{j})\\
& \leq & \eta_{\mathcal{D}}(r_1) \vee \cdots \vee \eta_{\mathcal{D}}(r_m)\vee \eta^k_{\mathcal{D}}(s_{1)} \vee \cdots \vee \eta^k_{A}(s_n)
\end{eqnarray*} 
Hence, $\lambda a + \phi b$ is another convex combination of points of $\mathcal{D}$. Therefore, the set of convex combinations of points of $A$ is itself a convex set that contains $\mathcal{D}$ which must coincide with the $Pch(\mathcal{D})$. 
\end{proof}

\ \\ \ \\
\section{Conclusion}
In this paper, we have introduced the notion of convexity of picture fuzzy multisets and some of their properties were obtained after studying the concept of picture fuzzy multisets

\ \\ \ \\

\end{document}